\documentclass{amsart}
\usepackage{latexsym}
\usepackage{amssymb}
\usepackage{amsfonts}
\usepackage{amsmath}
\usepackage{amsthm}
\usepackage{graphics}
\usepackage[pdftex]{graphicx}
\newtheorem{theorem}{Theorem}[section]
\newtheorem{lemma}[theorem]{Lemma}
\newtheorem{corollary}[theorem]{Corollary}

\newtheorem{definition}[theorem]{Definition}

\def\qed{\hfill$\Box$\vspace{12pt}}
\def\max{{\rm max}}
\def\diam{{\rm diam}}

\def\gcd{{\rm gcd}}

\def\mg{\mathbf{g}}

\def\De{\Delta}
\def\Ga{\Gamma}\def\max{{\rm max}}
\def\S_m{{\rm S_m}}

\def\gcd{{\rm gcd}}

\begin{document}
\title[Bipartite divisor graphs for integer subsets]{Bipartite divisor graphs for integer subsets}
\author[M. A. Iranmanesh]{Mohammad A. Iranmanesh}
\address{Mohammad A. Iranmanesh, Department of Mathematics, \newline Yazd University \\ Yazd, 89195-741 , Iran}
\email{iranmanesh@yazduni.ac.ir}
\author[C. E. Praeger]{Cheryl E. Praeger}
\address{Cheryl E. Praeger,  School of Mathematics and Statistics,\newline
The University of Western Australia,\\
 Crawley, WA 6009, Australia} \email{praeger@maths.uwa.edu.au}
\thanks{The first author wishes to thank Yazd University Research Council for financial support during his study leave, and The School of Mathematics and Statistics, University of Western Australia, in particular Prof. Praeger, for their hospitality during his visit and for the facilities and help provided.
The second author was supported by a Federation Fellowship of the Australian Research
Council.}
\begin{abstract}
Inspired by connections described in a recent paper by Mark L.
Lewis, between the common divisor graph $\Ga(X)$ and the prime
vertex graph $\De(X)$, for a set $X$ of positive integers, we define
the bipartite divisor graph $B(X)$, and  show that many of these
connections flow naturally from properties of $B(X)$. In particular
we establish links between  parameters of these three graphs, such
as number and diameter of components, and we characterise bipartite
graphs that can arise as $B(X)$ for some $X$. Also we  obtain
necessary and sufficient conditions, in terms of subconfigurations
of $B(X)$, for one $\Ga(X)$ or $\De(X)$ to contain a complete
subgraph of size $3$ or $4$.

\end{abstract}
\maketitle

%
\section{Introduction}
\label{sec:introd} We introduce the bipartite divisor graph $B(X)$,
for a non-empty subset $X$ of positive integers, that contains
information about two previously studied graphs, namely the prime
vertex graph and the common divisor graph for $X$. Our work was
inspired by a recent paper~\cite{Lewis} by Mark L. Lewis which
provides a fascinating overview of various graphs associated with
groups. Surprisingly strong combinatorial information available for
these graphs leads to structural information about the groups and
their representations. Lewis distilled and unified many results
concerning these `group graphs' (from the 78 references in his
bibliography) by first defining two graphs associated with an
arbitrary non-empty subset $X$ of positive integers:
\begin{itemize}
\item[(1)] the \emph{prime vertex graph} $\De(X)$ has, as vertex set $\rho(X)$, the set of
primes dividing some element of $X$, and two such primes $p,q$ are joined by an edge if and only if
$pq$ divides some $x\in X$;
\item[(2)] the \emph{common divisor graph} $\Ga(X)$  has vertex set
$X^*:=X\setminus\{1\}$, and $x,y\in X^*$ form an edge if and only if
$\gcd(x,y)>1$.
\end{itemize}
Although, in the group setting, the subset $X$ is usually the set of irreducible character degrees,
or the set of conjugacy class sizes, or conjugacy class indices,  of a finite group, Lewis
showed that, even for arbitrary integer sets $X$, the prime vertex graph and the common divisor
graph share very similar combinatorial properties, for example, they have the same number of
connected components, and similar diameters (where by the \emph{diameter} Lewis means the maximum
diameter of a connected component).

The \emph{bipartite divisor graph} $B(X)$, for an arbitrary
non-empty subset $X$ of positive integers, has as vertex set  the
disjoint union $\rho(X)\cup X^*$, and its edges are the pairs
$\{p,x\}$ where $p\in\rho(X)$, $x\in X^*$ and $p$ divides $x$. Thus
$B(X)$ is bipartite, and $\{\rho(X)|X^*\}$ forms a bi-partition of
the vertex set, that is unique if $B(X)$ is connected (and in the
group cases, $B(X)$ is often connected, see for
example~\cite[Theorems 4.1, 5.2, 7.1, 8.1, 9.8]{Lewis}, key
references being~\cite{Kazarin,Manz,ManzW,Palfy}). Moreover the
`distance $2$-graph', derived from $B(X)$ by replacing the edges of
$B(X)$ by the set of pairs $\{u, v\}$ that have distance $2$ in
$B(X)$, contains $\De(X)$ and $\Ga(X)$ as the subgraphs induced on
$\rho(X)$ and $X^*$, respectively.  Thus it is not surprising that
several combinatorial properties of $\De(X)$ and $\Ga(X)$ can be
derived from similar properties for $B(X)$. We study such properties
as the diameter, girth, number of connected components, and clique
number for these three graphs, obtaining precise links relating
these parameters for the various graphs, which we summarise below. Our
findings lead to interesting new questions in the `group case', some of
which are explored in a forthcoming paper~\cite {BDIP} of the authors
with Bubboloni and Dolfi.

In particular, in~\cite[Lemma 3.3]{Lewis}, Lewis proved that every
graph $\mathcal{G}$ is isomorphic to $\De(X)$ and to $\Ga(Y)$ for
some sets of positive integers $X$ and $Y$. Our main result
characterises those bipartite graphs that arise as $B(X)$ for some
$X$. The proof in Section~\ref{sec:2} gives an explicit construction
of a subset $X$, for a given bipartite graph $\mathcal{G}$.

\begin{theorem}\label{lem:2.7}
A bipartite graph $\mathcal{G}$ is isomorphic to $B(X)$, for some
non-empty set of positive integers $X$, if and only if $\mathcal{G}$ is non-empty
and has no isolated vertices.
\end{theorem}

\medskip\noindent{\bf Comment on Notation:}\quad

(a) We call a graph \emph{bipartite} if there is a bipartition
$\{V_1\ | V_2\}$ of its vertex set with both $V_1, V_2$ non-empty,
such that each edge joins a vertex of $V_1$ to a vertex of $V_2$. An
\emph{empty graph} is a graph with at least one vertex and no edges, and
a vertex in a graph is \emph{isolated} if it lies on no edge.

(b) Usually the set $X$ of positive integers is  clear from the context.
We therefore suppress $X$ in our notation and write $B, \De, \Ga$ for the graphs $B(X),
\De(X), \Ga(X)$ respectively. Similarly, for example, we denote by
$n(B), n(\De), n(\Ga)$ the number of connected components of
$B,\De,\Ga$ respectively, and, for  vertices $x,y$  in the same
connected component, we denote by $d_B(x,y), d_\De(x,y), d_\Ga(x,y)$
the distance (length of shortest path) between $x$ and $y$ for the
graph $B, \De, \Ga$ respectively. Following Lewis~\cite{Lewis}, we
define the \textit{diameter} as the maximum distance between
vertices in the same connected component, and we denote the
diameters of these three graphs by $\diam(B), \diam(\De),
\diam(\Ga)$ respectively. Also, if there is a cycle in the graph $B,
\De$ or $\Ga$, we denote the \textit{girth} (the length of the
shortest cycle) by $\mathbf{g}(B)$, $\mathbf{g}(\De)$, or
$\mathbf{g}(\Ga)$, respectively.

\medskip\noindent{\bf Summary of other results:}\quad
(Definitions of the additional graph theoretic concepts are given in the relevant subsection.)
\medskip

(1) $B, \De,\Ga$ have equal numbers of connected components, and
the maximum of $\diam(\De)$ and $\diam(\Ga)$ is
$\lfloor\frac{\diam(B)}{2}\rfloor$. \quad
(Lemma~\ref{lem:2.6})

(2) Any subset of $\{B,\De,\Ga\}$ may be acyclic, and the others not. However,
if $B$ contains a cycle of length greater than 4, then all three graphs contain cycles.\quad
(Lemma~\ref{lem:girth})

(3) Both $\De$ and $\Ga$ are acyclic if and only if each connected component of $B$ is a path or
isomorphic to $C_4$. \quad
(Theorem~\ref{trees})

(4) For $m=3,4$, at least one of $\De,\Ga$ contains a clique of size
$m$ (a subgraph $K_m$), if and only if $B$ contains a subgraph in a
specified list. \quad (Theorems~\ref{thm:2.11} and~\ref{thm:2.14})

\section{Representing bipartite graphs as $B(X)$}\label{sec:2}

In this section we prove Theorem~\ref{lem:2.7}, giving an explicit
construction of a subset $X$, for a given bipartite graph
$\mathcal{G}$. We illustrate the construction with a simple example
in Figure \ref{fig:figure1}.

\begin{figure}[here]
\begin{center}
\includegraphics[height=3.5cm]{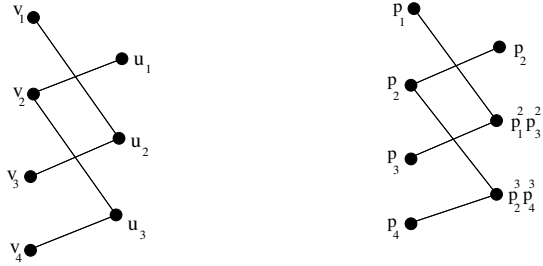}
\caption{Examples of $\mathcal{G}$ (left) and $B(X)$ (right) for
Lemma ~\ref{lem:2.7}} \label{fig:figure1}
\end{center}
\end{figure}

\par\noindent{\bf Proof of Theorem~\ref{lem:2.7}}
Suppose that  $\mathcal{G}$ is a bipartite graph with vertex bipartition $\{V_1|V_2\}$.
Let $V_1=\{ v_1, v_2, \dots, v_m\}$ and $V_2=\{u_1, u_2, \dots, u_n\}$ where $m\geq 1, n\geq1$.
Suppose first that $\mathcal{G}$ has no isolated vertices.
Let $p_1, p_2, \dots, p_m$ be pairwise distinct primes, and let
$M = \{p_1, p_2, \dots, p_m\}$. Define a bijection $\mathsf{f}: V_1
\longrightarrow M$ by $\mathsf{f}(v_i)=p_i$ for each $i$.
For $1\leq j\leq n$ define $I_j =\{\ell|\{v_\ell, u_j\}\in
E(\mathcal{G})\}$ and set $x_j=\prod_{\ell \in I_j}p_{\ell}^{j}$.
Note that $I_j\neq\emptyset$, since there are no isolated vertices in
$\mathcal{G}$. Let $X= \{x_j| 1\leq j\leq n\}$. The fact that $\rho(X)=M$
follows because there are no isolated vertices in
$\mathcal{G}$. Now $\{p_i, x_j\}\in E(B)$,
the edge set of $B= B(X)$, if and only if $p_i$ divides $x_j=\prod_{\ell
\in I_j} p_{\ell}^{j}$, that is, if and only if $i\in I_j$, and this holds
if and only if $\{v_i, u_j\}\in E(\mathcal{G})$. Thus extending $\mathsf{f}$ to a map $V(\mathcal{G})\longrightarrow V(B)$ by $\mathsf{f}(u_i)= x_i$, for each $i$, defines an isomorphism from $\mathcal{G}$ to $B(X)$.

Conversely, suppose that $\mathcal{G}\cong B(X)$, for some $X$.
Then, since by the definition of a bipartite graph, $\mathcal{G}$ has at least
one vertex, $X\neq\{1\}$. The fact that $B(X)$ has no isolated vertices now follows from its
definition. \qed

Theorem~\ref{lem:2.7} provides an important tool for the proofs in the rest of the paper, by
applying the following corollary.

\begin{corollary}\label{cor1}
 For a non-empty set $X$ of positive integers such that $X\ne\{1\}$, there exists a second
non-empty set $Y$ of positive integers, and a graph isomorphism $\phi:B(X)\rightarrow B(Y)$
that induces isomorphisms $\De(X)\cong \Ga(Y)$ and $\Ga(X)\cong\De(Y)$.
\end{corollary}

\begin{proof}
Let $\mathcal{G}=B(X)$ with vertex bipartition $\{\rho(X)\,|\,X^*\}$.
By definition, $\mathcal{G}$ is non-empty and has no isolated vertices.
We apply the proof of Theorem~\ref{lem:2.7} to the reverse
bipartition $\{X^*\,|\rho(X)\}$. This produces a non-empty subset $Y$ of
positive integers and a graph isomorphism $\phi:\mathcal{G} \rightarrow B(Y)$,
that induces a graph isomorphism from
$\De(Y)$ to the distance 2 graph induced on
the first part $X^*$ of the bipartition  (which by definition
of $\mathcal{G}$ is $\Ga(X)$), and a graph isomorphism from $\Ga(Y)$
to the distance 2 graph induced on the second part $\rho(X)$ of the
bipartition (which by definition of $\mathcal{G}$ is $\De(X)$).
\qed
\end{proof}
Thus if we wish to prove that a certain relationship holds between $B(X)$
and $\De(X)$, for all $X$, and also between  $B(X)$ and $\Ga(X)$, for
all $X$, it is often sufficient to prove only one of these assertions.

\section{Relating the parameters of $B, \De,\Ga$}\label{sect:params}

In this section we study certain parameters for the three graphs, namely distance, diameter, girth, and number of components. Throughout the section let $X$ denote a non-empty subset of positive integers with $X\ne\{1\}$, so that $X^*\ne\emptyset$. As mentioned above we simplify our notation and write $B:=B(X), \De:=\De(X), \Ga:=\Ga(X)$. We denote the vertex sets of these graphs by $V(B), V(\De), V(\Ga)$, and the edge sets by $E(B), E(\De), E(\Ga)$, respectively.

\subsection{Distance, diameter, and numbers of components}\label{distance}

A key technical result in Lewis's paper, namely~~\cite[Lemma 3.1 and
Corollary 3.2]{Lewis}, can be interpreted as a 1-1 correspondence
between the (connected) components of $\De$ and $\Ga$ leading to the
consequence that the diameters of $\De$ and $\Ga$ differ by at most
1. We extend these results to give analogous information about the
graph $B$ from which the facts about $\De$ and $\Ga$ may be deduced.
We note that, although in many of the `group cases' the graphs $\De$
and $\Ga$ have at most 2 components and diameter at most 3 (see for
example~\cite[Corollary 4.2, Theorems 7.1, 8.1, 8.3]{Lewis}), for
general $X$ these parameters may be arbitrarily large.

For $u\in V(B)$, let $[u]_B$ denote the connected component of $B$ containing $u$, and similarly define $[u]_\De, [u]_\Ga$ if $u\in V(\De)$ or $u\in V(\Ga)$ respectively.

\begin{lemma}\label{lem:2.6}
Let $p, q\in \rho(X)$ and $x, y\in X^*$ such that $[p]_B=[q]_B$ and $[x]_B=[y]_B$. Then,
\begin{itemize}
\item[{\em (a)}] $d_B(p, q)=2d_{\De}(p, q)$, $d_B(x, y)=2d_{\Ga}(x, y)$;

\item[{\em (b)}] if $p$ divides $x$ and $q$ divides $y$, then $[p]_B=[x]_B=[p]_\De\cup[x]_\Ga$ and $d_B(p, q)-d_B(x, y)\in \{-2,0,2\}$;

\item[{\em (c)}] $n(B)= n(\De)= n(\Ga)$;
\item[{\em (d)}] either
\begin{itemize}
\item[{\em(i)}]  $\diam(B)=2 \max \{\diam(\De), \diam(\Ga)\}$, and $|\diam(\De)-\diam(\Ga)|\leq 1$, or
\item[{\em(ii)}] $\diam(\De)=\diam(\Ga)=\frac{1}{2}(\diam(B)-1)$.

\end{itemize}

\end{itemize}

\end{lemma}

Table \ref{tab:2.1} gives simple examples to show that all possibilities for the diameters of $B, \De, \Ga$ given by Lemma \ref{lem:2.6}(d) arise.

\begin{table}
\caption{Illustration of all cases of Lemma \ref{lem:2.6}(d)}
\begin{center}
\begin{tabular}{|l||c||c||c|}
\multicolumn{4}{c}{}\\ \hline
\hspace{7.5mm}$X$&$\diam \De$&$\diam \Ga$&$\diam B$\\ \hline

$\{6, 10, 15\}$&$1$&$1$&$2$ \\ \hline

$\{2, 10\}$&$1$& $1$&$3$ \\ \hline

$\{2, 3, 6\}$&$1$&$2$&$4$ \\ \hline

$\{6, 15\}$&$2$&$1$&$4$ \\ \hline

\end{tabular}
\end{center}
\label{tab:2.1}
\end{table}

\medskip
\begin{proof}
(a) Suppose that $d_{\De}(p, q) = k$. Then there exists a shortest path $P_\De=(p_0, p_1, \dots, p_k)$ in $\De$ with $p= p_0$ and $q=p_k$. Now $\{ p_i, p_{i+1} \}$ is an edge of $\De$ if and only if $d_B (p_i, p_{i+1})=2$, and hence there exists a path $P_B=(p_0, x_1, p_1, \dots, x_{k-1}, p_k)$ in $B$ of length $2k$. Thus $d_B(p, q)\leq 2k$, and as $p,q$ are in the same part of the bipartition of $B$, we have $d_B(p, q)=2\ell\leq 2k$. If $P_B'=(p_0', x_1', p_1', \dots, x_{\ell}', p_{\ell}')$ is a shortest path in $B$ with $p= p_0'$ and $q=p_\ell'$, then $P_\De'=(p_0', p_1', \dots, p_{\ell}')$ is a path of length $\ell$  in $\De$, so $k=d_{\De}(p, q)\leq \ell$, and hence $d_B(p, q)=2k=2d_{\De}(p, q)$.
An analogous proof shows that $d_B(x, y)=2d_{\Ga}(x, y)$.

(b) Suppose now that $p$ divides $x$ and $q$ divides $y$. If the path $P_B'$ above can be chosen with $x_1'=x$ and $x_{\ell}'=y$ then $d_B(p,q)-d_B(x,y)\geq2$, and a similar argument to the above shows that equality holds; on the other hand if one of these equalities, but not the other can be achieved then we find $d_B(p,q)=d_B(x,y)$, while if neither can hold on a shortest path $P_B'$, then $d_B(x,y)-d_B(p,q)=2$. This proves that $d_B(p, q)-d_B(x, y)\in \{-2,0,2\}$,
and (see Table~\ref{tab:2.1}) all cases are possible. Moreover in this case, by assumption, the component $[p]_B$ contains all of $p,q,x,y$, and the fact that  $[p]_B=[p]_\De\cup[x]_\Ga$ follows from the proof of part (a).

(c) The assertion about numbers of connected components follows from the assertion in part (b) about components.

(d) Let $m=\max\{\diam(\De), \diam(\Ga)\}$. Then it follows from part (a) that $\diam(B)\geq 2m$. Let $M=\diam(B)$, so that $M\geq 2m$, and choose $a,b\in V(B)$ such that $d_B(a,b)=M$. If $a$ and $b$ are both in $\rho(X)$ (or both in $X^*$), then $M= d_B(a, b)=2d_{\De}(a, b)\leq2\diam (\De)\leq 2m$ (respectively, $M \leq 2 \diam (\Ga)\leq 2m$) and in either case we conclude that $M=2m$. On the other hand (without loss of generality) suppose that $a\in {\rho(X)}$ and $b\in X^*$, so $M$ is odd and hence $M\geq 2m+1$. Let $p\in\rho(X)$ be the vertex adjacent to $b$ on a path $P_B$ of length $M$ from $a$ to $b$. Then
the sub-path of $P_B$, from $p$ to $a$, must be a shortest path between these two vertices, by definition of $M$, and hence $M-1 = d_B(a, p) = 2 d_{\De}(a, p)$, by part (a), and this is at most $2\diam(\De)$. A similar argument using the vertex adjacent to $a$ on $P_B$ yields
$M \leq2 \diam(\Ga)+1$. It follows that $\diam(\De)=\diam(\Ga)=\frac{M-1}{2}$.

To prove the last assertion of (i) we may assume that $\diam(B)=2m$.
Let $\jmath=\diam(\De)$ and let $p_0, p_\jmath\in \rho(X)$ be such
that $d_\De(p_0,p_\jmath)=\jmath$. Then by part (a), there exists a
path $P_B'$ of length $2\jmath$ in $B$ from $p_0$ to $p_\jmath$. Let
$x_0, x_1$ be vertices on $P_B'$ adjacent to $p_0$ and $p_\jmath$
respectively. Then the sub-path of $P_B'$ from  $x_0$ to $x_1$ of
length $2\jmath-2$ is a shortest path in $B$ between these two
vertices. Thus $d_B(x_0, x_1)= 2\jmath-2$. By part (a),
$d_{\Ga}(x_0, x_1)= \jmath-1$ and therefore $\diam(\Ga)\geq
\diam(\De)-1$. A similar argument shows that $\diam(\De)\geq
\diam(\Ga)-1$. Hence $|\diam(\De)-\diam(\Ga)|\leq 1$.\qed
\end{proof}

\subsection{Cycles and girth}\label{girth}
A graph $\mathcal{G}$ is said to be \emph{acyclic} if it contains no
cycles, that is, it contains no closed paths of length at least 3.
On the other hand, recall that, if $\mathcal{G}$ contains a cycle
then the minimum length of its cycles is called its \emph{girth} and
denoted $\mg(\mathcal{G})$. For any subset $\mathbf{K}\subseteq
\{B,\De,\Ga\}$ there exists $X$ such that the graphs in $\mathbf{K}$
are acyclic and each graph not in $\mathbf{K}$ contains a cycle.
Examples of subsets $X$ are provided in Table~\ref{tab:girth} for
the seven non-empty subsets $\mathbf{K}$ of $\{B,\De,\Ga\}$, and if
$X=X_2\cup X_3\cup X_4$, with the $X_i$ as in Table~\ref{tab:girth},
then all three graphs contain cycles. In this last example, $B$ has
girth 4, while the other two graphs have girth 3. However, once the
graph $B$ contains a cycle of length greater than 4, we prove that
all three of the graphs contain cycles. Even in this case it is
possible for one or both of $\De$ or $\Ga$ to have girth 3
regardless of the size of $\mg(B)$, simply by adding to $X$ an
analogue of the subset $X_2$ or $X_3$ of Table~\ref{tab:girth}
(involving suitable primes). However if the girths of $\De$ or $\Ga$
are greater than 3, we show that there is a tight link between these
girths and the minimum length of cycles in $B$ with more than 4
vertices.

In Table~\ref{tab:girth} we denote a complete graph and a cycle on $m$ vertices by $K_m$ and $C_m$, respectively, and if $B=B(X)$ is a complete bipartite graph with $|\rho(X)|=m$ and $|X^*|=n$, then we denote $B$ by $K_{m,n}^\rightarrow$.

\begin{table}
\caption{Illustration of acyclic possibilities for $B, \De, \Ga$}
\begin{center}
\begin{tabular}{|l|c|c|c|c|}
\hline
$i$&$X_i$&$B$&$\De$&$\Ga$\\ \hline
$1$&$\{2\}$&$K_2$&$K_1$&$K_1$ \\
$2$&$\{2,4,8\}$&$K_{1,3}^\rightarrow$&$K_1$&$K_3$ \\
$3$&$\{105\}$ &$K_{3,1}^\rightarrow$&$K_3$&$K_1$ \\
$4$&$\{11\cdot13,11^2\cdot13\}$&$C_4$&$K_2$&$K_2$ \\
$5$&$X_2\cup X_3$&$K_{1,3}^\rightarrow+K_{3,1}^\rightarrow$&
$K_1+K_3$&$K_1+K_3$ \\
$6$&$X_2\cup X_4$&$K_{1,3}^\rightarrow+C_4$&
$K_1+K_2$&$K_2+K_3$ \\
$7$&$X_3\cup X_4$&$K_{3,1}^\rightarrow+C_4$&
$K_2+K_3$&$K_1+K_2$ \\  \hline
\end{tabular}
\end{center}
\label{tab:girth}
\end{table}

\begin{lemma}\label{lem:girth}
Suppose that $B$ contains a cycle of length greater than $4$. Then each of
$\De$ and $\Ga$ also contains a cycle. Moreover, for $\Phi\in\{\De,\Ga\}$, either $\mathbf{g}(\Phi)= 3$ or  $\mathbf{g}(\Phi)= \frac{1}{2}
\mathbf{g}'(B)$, where $\mathbf{g}'(B)$ is the minimum length  of cycles of $B$ with more than four vertices.
\end{lemma}

\begin{proof}
Since $B$ is bipartite, $\mathbf{g}'(B)=2k$ for some $k\geq3$.
Let $P_B=(p_1, x_1, \dots, p_k, x_k)$ be a closed path of length $2k$ in $B$ with
the $p_i\in\rho(X)$ and the $x_i\in X^*$. By the definition of $B$, $p_i$
divides $x_i$ and $x_{i-1}$, for $i=1, \dots, k$, reading the subscripts modulo $k$.
Hence there exist closed paths of length $k$ in both $\De$ and $\Ga$. This implies
that both $\De$ and $\Ga$ contain cycles and $\mathbf{g}(\De)\leq k, \mathbf{g}(\Ga)\leq k$.

If $\mathbf{g}(\De)=\ell<k$, then there exists a
closed path $P_\De=(p_1', p_2', \dots, p_{\ell}')$ in $\De$.
By the definition of $\De$, for each $i$, there exists $x_i'\in X^*$
that is divisible by both $p_i'$ and $p_{i+1}'$, reading subscripts modulo
$\ell$. If the $x_i'$ are pairwise distinct, then  $P_B'=(p_1', x_1', \dots, p_\ell', x_\ell')$
is a closed path in $B$ of length $2\ell$, and $6\leq 2\ell < 2k=\mg'(B)$, which
is a contradiction. Hence the $x_i'$ are not all distinct. Let $i,j$ be such
that $1\leq i<j\leq\ell$ and $x_i'=x_j'$. Then in $\De$ the induced subgraph on
the subset $\{p_i',p_{i+1}',p_j',p_{j+1}'\}$ is a complete graph
(of order 3 or 4) and hence $\ell=\mg(\De)=3$. Thus either $\mg(\De)=3$ or $\mg(\De)=k$. A similar proof shows that either $\mg(\Ga)=3$ or $\mg(\Ga)=k$. \qed
\end{proof}

We consider further, in Section ~\ref{sub:acyclic}, the case where both $\Ga$ and
$\De$ are acyclic, characterising the graphs $B$ in this case.

\section{Subgraphs of $B, \De, \Ga$}\label{sect:configs}

In this section we prove several results that link existence of certain
subgraphs in $B$ with the existence of related subgraphs in
$\De$ and $\Ga$.
Let  $\mathcal{G}=(V, E)$ be a graph with vertex set $V$
and edge set $E$. By a \emph{subgraph} of $\mathcal{G}$, we mean a graph
$\mathcal{G}_0=(V_0, E_0)$ where $V_0\subseteq V$ and $E_0\subseteq E \cap
V_0^{\{2\}}$. If $E_0=E\cap V_0^{\{2\}}$, then $\mathcal{G}_0$ is called an
\emph{induced subgraph}.

\subsection{Triangles in $\De$ and $\Ga$}\label{sub:triangles}

First we look at the existence of triangles (that is, 3-cycles, closed paths
of length 3) in the graphs $\De$ and $\Ga$.

\begin{theorem}\label{thm:2.11}
At least one of $\De, \Ga$ contains a triangle if and only if $B$
contains $C_6$ or $K_{1,3}$ as an induced subgraph.
\end{theorem}
\begin{proof}
Suppose first that $\mathbf{g}(\Ga)=3$ and let $P_\Ga=(x_1, x_2, x_3)$
be a cycle in $\Ga$. If there exists a prime $p$ which divides $x_i$,
for all $i=1, 2, 3$, then the set $\{p, x_1, x_2, x_3 \}$ induces a
subgraph $K_{1,3}$ of $B$. So we may assume that no such prime exists.
Then, since $P_\Ga$ is a cycle in $\Ga$, there are distinct primes $p_1, p_2,
p_3$ such that, for each $i$, $p_i$ divides $x_{i-1}$ and $x_i$, writing
subscripts modulo 3, and the set $\{p_1, x_1, p_2, x_2, p_3, x_3\}$ induces
a subgraph $C_6$ of $B$. Thus $\mathbf{g}(\Ga)=3$ implies that $B$
contains  an induced subgraph isomorphic to either $C_6$ or $K_{1,3}$.
By Corollary~\ref{cor1}, it follows that $\mathbf{g}(\De)=3$ implies that $B$
contains  an induced subgraph isomorphic to either $C_6$ or $K_{1,3}$.

Conversely, if $\{p_1, x_1, p_2, x_2, p_3, x_3\}$ induces a subgraph $C_6$
in $B$, where the
$p_i\in\rho(X)$ and the $x_i\in X^*$, then $(p_1, p_2, p_3)$ and $(x_1,
x_2, x_3)$ are cycles in $\De$ and $\Ga$ respectively, so $\mg(\De)
=\mg(\Ga)=3$.  Similarly if $B$ contains an induced subgraph $K_{1,3}$,
then at least one of $\De$, $\Ga$ contains a triangle.
This completes the proof. \qed
\end{proof}

\subsection{Acyclic graphs}\label{sub:acyclic}

Next we characterise the cases where both $\De$ and $\Ga$ are acyclic.

\begin{theorem}\label{trees}
Both the graphs $\Ga$ and $\De$ are acyclic if and only if each connected component
of $B$ is a path or a cycle of length $4$.
\end{theorem}

\begin{proof}
Suppose first that $\De,\Ga$ are both acyclic. If some vertex of $B$ lies on at least three edges, then
one of $\De$, $\Ga$ contains a 3-cycle, which is a contradiction.
Thus each vertex of $B$ lies on at most two edges in $B$. Since $B$ is bipartite, this means that
each connected component of $B$ is a path, or a cycle $C_{2k}$ of even length $2k\geq4$. Moreover, in
the case of a component $C_{2k}$, it follows from Lemma~\ref{lem:girth} that $k=2$.

Conversely, suppose that each component of $B$ is a path or isomorphic to $C_4$. For a component
$C_4$ of $B$, the corresponding component of $\De, \Ga$ is isomorphic to $K_2$.
Consider a component $B'$ of $B$ which is a path. Suppose that $P_\De=(p_1, p_2, \dots,
p_{\ell})$ is a cycle in the corresponding component of $\De$ of length $\ell\geq3$.
By the definition of $\De$, for each $i$, there exists $x_i\in X^*$
that is divisible by both $p_i$ and $p_{i+1}$, reading subscripts modulo
$\ell$. If the $x_i$ are pairwise distinct, then  $P_{B'}=(p_1, x_1, \dots, p_\ell,
x_\ell)$ is a cycle in $B'$, which is a contradiction.
Hence there exist $i,j$ such
that $1\leq i<j\leq\ell$ and $x_i=x_j$. This however implies that $x_i$ is
joined to at least three vertices in $B'$,
contradicting the fact that $B'$ is a path.
Hence the  component of $\De$ corresponding to $B'$  is acyclic. A similar proof shows
that the component of $\Ga$ corresponding to $B'$ is also acyclic.
\qed \end{proof}

We have the following immediate corollary for the case where both $\De$ and $\Ga$ are trees, where
by a \emph{tree} we mean a connected acyclic graph.

\begin{corollary}
Both graphs $\Ga$ and $\De$ are trees if and only if either $B$ is a path
or $B\cong C_4$.
\end{corollary}

\subsection{Incidence graphs of complete graphs}\label{inc}

As preparation for our final theorem, we study the
existence of incidence graphs of complete graphs as subgraphs of $B$.

\begin{definition}\label{def:2.12}{\rm
Let $\mathcal{G}=(V, E)$ be a graph with vertex set $V$ and edge set $E$.
Then the \emph{incidence graph} $\mathtt{Inc}(\mathcal{G})$ of $\mathcal{G}$,
is the bipartite graph with vertex set $V\dot{\cup}E$ such that
$\{v, e\}$ forms an edge if and only if $v\in V$, $e\in E$ and $v$ is
incident with $e$ in $\mathcal{G}$.}
\end{definition}

\begin{lemma}\label{lem:2.13}
 The graph $B$ contains a subgraph isomorphic to $\mathtt{Inc}(K_{\ell})$
if and only if one of the following conditions {\em (i)} or {\em (ii)} holds.

{\em (i)} $\Ga$ contains a complete subgraph $K_{\ell}$ with vertices $\{x_1, x_2, \dots, x_{\ell}\}$, and there are $\binom{\ell}{2}$ pairwise distinct primes $p_{ij}$, for $1\leq i< j\leq \ell$, such that $p_{ij}$ divides $\gcd(x_i, x_j)$.

{\em (ii)} $\De$ contains a complete subgraph $K_{\ell}$ with vertices $\{p_1, p_2, \dots, p_{\ell}\}$ and there are $\binom{\ell}{2}$ pairwise distinct numbers $x_{ij}\in X^*$, for $1\leq i< j\leq \ell$, such that $p_ip_j$ divides $x_{ij}$.

\end{lemma}

\begin{proof}
Suppose that condition (i) holds and let $\mathcal{G}$ be the subgraph of $B$
with $V(\mathcal{G})=\{x_1, x_2, \dots, x_{\ell}\}\dot{\cup}\{p_{ij}, 1\leq i<j\leq \ell \}$ and edges $\{p_{ij}, x_i\}$ and $\{p_{ij}, x_j\}$ for each $i, j$. (Note that $\mathcal{G}$ may not be an induced subgraph.) Let $K_{\ell}$ be the complete graph on $V=\{1, 2, \dots, \ell\}$ with edge set $E=\{e_{ij}=\{i, j\}| 1\leq i<j\leq \ell \}$. Define $\Psi:\mathtt{Inc}(K_{\ell})\longrightarrow \mathcal{G}$ by $\Psi:i\longrightarrow x_i$, $e_{ij}\longrightarrow p_{ij}$. It is straightforward to check that $\Psi$ is a graph isomorphism. A similar isomorphism can be constructed if condition (ii) holds.

Conversely suppose that $B$ contains a subgraph $\mathcal{M}$ isomorphic to
$\mathtt{Inc}(K_{\ell})$. Then $\mathcal{M}$ is connected and bipartite, and
hence one of its bipartite halves, say $V$, has size $\ell$, and each pair of
its elements are at distance two in $B$. Moreover $V$ must be contained in
$X^*$ or $\rho(X)$,  and hence induce a complete subgraph $K_{\ell}$ of
$\Ga$ or $\De$, respectively. It is straightforward to check the
remaining assertions of the condition (i) or (ii) respectively. \qed
\end{proof}

\subsection{Complete subgraphs $K_4$ of $\De$ and $\Ga$}\label{k4}

In this final subsection we show in Theorem~\ref{thm:2.14} that
existence of a complete subgraph $K_4$
of $\De$ or $\Ga$ is equivalent to existence of at least one of a small
number of possible subgraphs of $B$.

To demonstrate that each of the cases
of Theorem ~\ref{thm:2.14} does indeed occur,  we give in
Tables ~\ref{tab:2.2} and ~\ref{tab:2.3}
examples of small subsets $X$ for the various cases.
In Table ~\ref{tab:2.2}, $p, p_1, \dots, p_6$ denote primes such that
$p_i\neq p_j$ for $i\neq j$. Also $L(K_4)$ denotes the \textit{line graph}
of $K_4$, with vertex set $E(K_4)$ and two vertices adjacent if and only if
the corresponding edges of $K_4$ have a vertex in common. We denote by $K_{a, b}^\rightarrow$ a  complete
bipartite subgraph of $B(X)$ with $a$ vertices in $\rho(X)$ and $b$ vertices
in $X^*$. Finally, recall the definition of $\mathtt{Inc}(K_{\ell})$ from
Definition~\ref{def:2.12}, and  let $\mathcal{K}$, $\mathcal{G}$ denote the
graphs in Figure ~\ref{fig:figure2}.

\begin{table}
\caption{Small examples for each case of Theorem~\ref{thm:2.14} with
$\Ga(X)=K_4$.}
\label{tab:2.2}
\begin{center}
\begin{tabular}{|l||c||c||c|}
\multicolumn{4}{c}{} \\ \hline \hspace{25mm}$X$&$B$&$\Ga$&$\De$\\
\hline $\{p, p^2, p^3, p^4\}$& $K_{1,4}^\rightarrow$&$K_4$& $K_1$ \\
\hline $\{p_1 p_2, p_1^2p_2, p_1p_3, p_2p_3\}$&$\mathcal{K}$&
$K_4$&$K_3$ \\ \hline $\{p_1p_2, p_1p_3, p_1p_4,
p_2 p_3 p_4\}$&$\mathcal{G}$&$K_4$&$K_4$ \\ \hline $\{p_1 p_2 p_3,
p_1 p_4 p_5, p_2 p_4 p_6, p_3 p_5 p_6\}$&$\mathtt{Inc}({K_4})$&$K_4$
&$L(K_4)$ \\ \hline
\end{tabular}
\end{center}
\end{table}

\begin{table}
\caption{Small examples for each case of Theorem \ref{thm:2.14} with
$\De(X)=K_4$.}
\begin{center}
\begin{tabular}{|l||c||c||c|}
\multicolumn{4}{c}{} \\ \hline
\hspace{25mm}$X$&$B(X)$&$\Ga(X)$&$\De(X)$\\ \hline

$\{p_1 p_2 p_3 p_4\}$& $K_{4,1}^\rightarrow$&$K_1$&$K_4$  \\ \hline

$\{p_1p_2p_3, p_1p_4, p_2p_3p_4\}$&$\mathcal{K}$& $K_3$&$K_4$ \\ \hline

$\{p_1p_2p_3, p_2p_4, p_3p_4, p_1p_4\}$&$\mathcal{G}$&$K_4$&$K_4$ \\ \hline

$\{p_1p_2, p_1p_3, p_1p_4, p_2p_3,p_2p_4, p_3p_4\}$&$\mathtt{Inc}(K_4)$&$L(K_4)$ &$K_4$\\ \hline

\end{tabular}
\end{center}
\label{tab:2.3}
\end{table}

\begin{figure}
\begin{center}
\includegraphics[height=3cm]{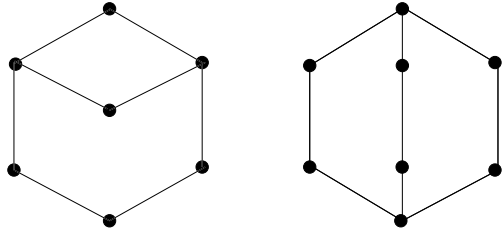}
\caption{The graph $\mathcal{K}$ (to the left) and
$\mathcal{G}$ (to the right) of Theorem ~\ref{thm:2.14}}
\label{fig:figure2}
\end{center}
\end{figure}

\begin{theorem}\label{thm:2.14}
{\em (i)} If $\De$ has a subgraph $K_4$, then $B$ contains a subgraph
isomorphic to one of $K_{4,1}^\rightarrow$, $\mathtt{Inc}(K_4)$, $\mathcal{K}$
or $\mathcal{G}$.

{\em (ii)} If $\Ga$ has a subgraph $K_4$, then $B$ contains a subgraph
isomorphic to one of $K_{1,4}^\rightarrow$, $\mathtt{Inc}(K_4)$, $\mathcal{K}$
or $\mathcal{G}$.

{\em (iii)} If $B$ contains a subgraph isomorphic to one of
$K_{1,4}^\rightarrow$, $K_{4,1}^\rightarrow$, $\mathtt{Inc}(K_4)$,
$\mathcal{K}$ or $\mathcal{G}$, then at least one of $\De$ or $\Ga$
has a subgraph $K_4$.
\end{theorem}

\begin{proof}
(i) Suppose that $\pi=\{p_1, p_2, p_3, p_4\}\subseteq \rho(X)$ induces
a subgraph $K_4$ of $\De$. If there exists $x\in X$ divisible by
$\prod_{i=1}^4p_i$, then the subgraph of $B$ induced on $\pi\cup\{x\}$ is
$K_{4,1}^\rightarrow$. Thus we may assume that no such $x$ exists.
By the definition of $\De$, for each $i,j$ satisfying $1\leq i < j\leq 4$,
there exists $x_{ij}\in X$ such that $p_ip_j$ divides $x_{ij}$.

Suppose next that some element $x\in X$ is divisible by three of the $p_i$,
without loss of generality, that $x$ is divisible by $p_1p_2p_3$. If $x_{14},
x_{24}, x_{34}$ are all distinct, then the subgraph of $B$ induced on
$\pi\cup\{x, x_{14}, x_{24}, x_{34}\}$ contains the graph $\mathcal{G}$
of Figure~\ref{fig:figure2}. If this is
not the case then, without loss of generality, $x_{14}=x_{24}$, and this
number is therefore divisible by $p_1p_2p_4$. By our assumption, $p_3$
does not divide $x_{14}$, and hence $x_{34}\ne x_{14}$, and the
subgraph of $B$ induced on $\pi\cup\{x, x_{14}, x_{34}\}$ contains
the graph $\mathcal{K}$ of Figure~\ref{fig:figure2}.

Thus we may assume that no element of $X$ is divisible by more than two
primes in $\pi$, and hence that
the $x_{ij}$ are pairwise distinct. It now follows
from Lemma~\ref{lem:2.13} that $B$ contains a subgraph isomorphic to
$\mathtt{Inc}(K_4)$.

(ii) Suppose that $X_0\{x_1, x_2, x_3, x_4\}\subseteq X^*$ induces
a subgraph $K_4$ of $\Ga$. By Corollary~\ref{cor1}, there is a set $Y$
and a graph isomorphism $\phi:B(X)\rightarrow B(Y)$ that
induces an isomorphism $\Ga(X)\cong \De(Y)$.
Thus $\De(Y)$ has an induced subgraph $K_4$, and hence part (ii)
follows from part (i).

(iii) Finally suppose that $B$ contains a subgraph $\mathcal{H}$
isomorphic to one of $K_{1,4}^\rightarrow$, $K_{4,1}^\rightarrow$,
$\mathtt{Inc}(K_4)$, $\mathcal{K}$ or $\mathcal{G}$. Then $\mathcal{H}$ is
connected and bipartite, and the distance 2 graph induced on one of its
bipartite halves is isomorphic to $K_4$. Thus $\De$ or $\Ga$ has a subgraph
isomorphic to $K_4$. \qed
\end{proof}



\begin{thebibliography}{}
\bibitem{BDIP}
D.~Bubboloni, S.~Dolfi, M.~A.~Iranmanesh, C.~E.~Praeger, On bipartite divisor graphs for group conjugacy class sizes, Journal of Pure and Applied Algebra. {\textbf 213} (2009), 1722-1734.

\bibitem{Kazarin}
L.~S.~Kazarin, On groups with isolated conjugacy classes. (Russian), Izv. Vyssh. Uchebn. Zaved. Mat. (1981) no. 7, 40-45 (English Translation: Soviet Math. (Iz VUZ)) {\textbf 25} (1981), 43-49.

\bibitem{Lewis}
Mark~L.~Lewis, An overview of graphs associated with character degrees and conjugacy class sizes in finite groups. Rocky Mountain J. Math. {\textbf 38} (2008), 175-212.

\bibitem{Manz}
O.~Manz, Degree problems II: $\pi$-separable character degrees, Comm. Algebra {\textbf 13} (1985), 2421-2431.

\bibitem{ManzW}
O.~Manz and T.~R.~Wolf, Representations of solvable groups, Cambridge Univ. Press, Cambridge, 1993.

\bibitem{Palfy}
P.~P.~P$\check{a}$lfy, On the character degree graph of solvable groups, II: disconnected graphs, Studia Sci. Math. Hungar. {\textbf 38} (2001), 339-355.

\end{thebibliography}
\end{document}